\documentclass[12pt,letterpaper]{amsart}
\usepackage{times}
 \usepackage{stmaryrd} 
 \usepackage{mathrsfs}
 \usepackage{amsmath}
 \usepackage{amsthm}
 \usepackage{amsfonts}
 \usepackage{amssymb}
 \usepackage{amsxtra}
 \usepackage[dvips]{graphicx}
 \usepackage{wrapfig}
 \usepackage{layout}
 \usepackage{alltt}
 \usepackage{xypic}
 \input xy
 \xyoption{all}

\SelectTips{cm}{10}

\usepackage[pdftex,
                paper=a4paper,
                portrait=true,
                textwidth=387pt,
                textheight=608pt,
                tmargin=3cm,
                marginratio=1:1]{geometry}

\def\Cx{{\mathbb C}}

\def\Px{{\mathbb P}}

\def\Zx{{\mathbb Z}}

\def\Sg{{\rm Sg}}
\def\gr{{\rm gr}}
\def\id{{\rm id}}

\newcommand{\As}{\mathscr{A}}

\newcommand{\Cs}{\mathscr{C}}

\newcommand{\Es}{\mathscr{E}}
\newcommand{\Fs}{\mathscr{F}}
\newcommand{\Gs}{\mathscr{G}}
\newcommand{\Hs}{\mathscr{H}}

\newcommand{\Ks}{\mathscr{K}}

\newcommand{\Os}{\mathscr{O}}

\newcommand{\Rs}{\mathscr{R}}
\newcommand{\Ss}{\mathscr{S}}

\newcommand{\Zs}{\mathscr{Z}}

\newcommand{\Ld}{\mathbf{L}}
\newcommand{\Rd}{\mathbf{R}}
\newcommand{\Dd}{\mathbf{D}}
\newcommand{\bHom}{\mathbf{Hom}}

\newcommand{\fl}{\rightarrow}
\newcommand{\lfl}{\longrightarrow}

\newcommand\uMod{\operatorname{Mod \textendash \!}}
\newcommand\umod{\operatorname{mod \textendash \!}}
\newcommand\ueMod{\operatorname{Mod}}
\newcommand\uemod{\operatorname{mod}}

\newcommand\uQcoh{\operatorname{Qcoh}}
\newcommand\uCoh{\operatorname{Coh}}
\newcommand\uExt{\operatorname{Ext}}
\newcommand\uAut{\operatorname{Aut}}

\newcommand\uMF{\operatorname{MF}}

\newcommand\uPerf{\operatorname{Perf}}
\newcommand\uPair{\operatorname{Pair}}
\newcommand\uCok{\operatorname{Cok}}

\newcommand\uHilb{\operatorname{Hilb}}
\newcommand\uHom{\operatorname{Hom}}

\newcommand\ucoker{\operatorname{coker}}

\newcommand\uSpec{\operatorname{Spec}}
\newcommand\uSing{\operatorname{Sing}}

\newcommand\uSL{\operatorname{SL}}

\newcommand\uU{\operatorname{U}}

\newcommand\utot{\operatorname{tot}}
\newcommand\ui{\operatorname{i}}
\newcommand\ud{\operatorname{d \!}}

\theoremstyle{plain}
 \newtheorem{thm}{Theorem}[section]
 \newtheorem{lem}[thm]{Lemma}
 \newtheorem{prop}[thm]{Proposition}
\newtheorem{assum}[thm]{Assumption}
 \newtheorem{cor}[thm]{Corollary}
 \theoremstyle{definition}
 \newtheorem{dfn}[thm]{Definition}
 
\newtheorem{assertion}[thm]{Assertion}
 \newtheorem{princ}[thm]{Principle}
\newtheorem*{ack}{{\footnotesize Acknowledgments}}
 \theoremstyle{remark}
 \newtheorem*{note}{Note added}
\newtheorem*{add}{Addendum}

\author{Alexander Quintero V\'{e}lez}
\email{quintero@math.uu.nl}
\address{Mathematisch Instituut\\
Universiteit Utrecht\\
P.O.~Box 80.010, NL-3508 TA Utrecht\\
Nederland}

\title{McKay correspondence for Landau-Ginzburg models}
\begin{document}
\begin{abstract} 
In this paper we prove an analogue of the McKay correspondence for Landau-Ginzburg models. Our proof is based on the ideas introduced by T.~Bridgeland, A.~King and M.~Reid, which reformulate and generalize the McKay correspondence in the language of derived categories, along with the techniques introduced by J.-C.~Chen.
\end{abstract}

\maketitle

\section{Introduction}
The goal of this paper is to describe an analogue of the McKay correspondence for Landau-Ginzburg models. Before going into details, it is useful to review some aspects of the McKay correspondence that are relevant for our considerations.

In its original form, the McKay correspondence was observed as a nice relation between the irreducible representations of a finite subgroup $G$ of $\uSL(2,\Cx)$ on the one hand, and the geometry of the exceptional divisor in a minimal resolution of $\Cx^2/G$ on the other hand (cf~\cite{McK80}). The first hint of a McKay correspondence in higher dimensions came from the work of L.~Dixon, J.~Harvey, C.~Vafa and E.~Witten. It was conjectured in \cite{DHVW86} that for a finite subgroup $G \subset \uSL(n,\Cx)$ acting on $\Cx^n$, the Euler characteristic of a crepant resolution $Y$ of the quotient space $\Cx^n/G$ equals the number of conjugacy classes, or equivalently the number of equivalence classes of irreducible representations of $G$. If $n=2$, the equality can be viewed as a version of the McKay correspondence. As a result, this formula may be regarded as a generalization of the McKay correspondence to an arbitrary dimension $n$. The McKay correspondence became recently a subject of intense study in both physics and mathematics. However, the term is now primarily used to indicate a relationship between the various invariants of the actions of finite automorphism groups on quasiprojective varieties and resolutions of the corresponding quotients by such actions.

The guiding principle behind the McKay correspondence was stated by M. Reid along the following lines:

\begin{princ}\label{ReidPrinciple}
Let $M$ be an algebraic variety, $G$ a group of automorphisms of $M$, and $Y$ a crepant resolution of singularities of $X=M/G$. Then the answer to any well posed question about the geometry of $Y$ is the $G$-equivariant geometry of $M$.
\end{princ}

Applied to the case of quotient singularities $X=\Cx^n/G$ arising from a finite subgroup $G \subset \uSL(n,\Cx)$, the content of this slogan is that the $G$-equivariant geometry of $M=\Cx^n$ already {\it knows} about the crepant resolution $Y$. In particular, any two crepant resolutions of $X$ should have equivalent geometries.

Reid suggested that one manifestation of Principle \ref{ReidPrinciple}  should be a derived equivalence $\Dd(Y) \cong \Dd^G(M)$, where $\Dd(Y)$ is the bounded derived category of coherent sheaves on $Y$ and $\Dd^G(M)$ is the bounded derived category of $G$-equivariant coherent sheaves on $M$. This has been worked out by Kapranov and Vasserot \cite{KV00} in dimension $n=2$ and generalized to higher dimensions including all cases of finite subgroups of $\uSL(3,\Cx)$ by Bridgeland, King and Reid \cite{BKR01}. In the latter case the quotient singularity $X=\Cx^3/G$ always has a crepant resolution, a distinguished choice being given by the Hilbert scheme of $G$-orbits $G\text{-}\!\uHilb(M)$. This scheme is perhaps best thought of as a moduli space of representations of the skew group algebra $A=\Cx[x,y,z] * G$ that are stable with respect to a certain choice of stability condition. Indeed, this is closely related to the physicist's understanding of D-branes as objects in the derived category. 

In string theory, space-time $X$ is represented by a two-dimensional quantum field theory with $N=2$ supersymmetry. A quite important class of such theories are nonlinear sigma models on a K\"{a}hler manifold $X$. In this case, E.~Witten explained how to manufacture two dimensional topological field theories. He showed that any nonlinear sigma model with a K\"{a}hler target space $X$ admits a topologically twisted version called the A-model; if $X$ is a Calabi-Yau manifold, there is another topologically twisted theory, the B-model. A similar construction exists in the equivariant setting. Given an action of a finite group $G$ on a space $X$ satisfying certain properties, one can construct a two-dimensional topological field theory which represents the $G$-equivariant physics of $X$. To be more precise, one associates a $G$-gauged sigma model to a presentation of the quotient stack $[X/G]$:~the gauged sigma model can be interpreted as a sigma model on $[X/G]$.

Open strings are associated to extended objects, different from strings, which go under the name of D-branes. Loosely speaking, a D-brane is a  \textquoteleft nice\textquoteright~boundary condition for the two-dimensional quantum field theory. To any topologically twisted sigma model one can associate a category of D-branes. In the case of the topological B-model of a Calabi-Yau $X$, the category of D-branes is believed to be equivalent to the bounded derived category $\Dd(X)$ of coherent sheaves on $X$. In the equivariant setting this should be replaced by the bounded derived category $\Dd([X/G]) \cong \Dd^G(X)$ of $G$-equivariant coherent sheaves on $X$.

From the previous consideration we see that the McKay correspondence has a completely natural explanation in terms of nonlinear sigma models with boundaries. Indeed, arguments from topological open string theory, formalized in the \textquoteleft decoupling statement\textquoteright~of \cite{BDLR00}, suggest that there is an equivalence $\Dd(Y) \cong \Dd([M/G])$ for any crepant resolution $Y$ of the singularities of $X=M/G$.

In this paper we study another class of topological field theories: topological Landau-Ginzburg models. The general definition of a Landau-Ginzburg model involves, besides a choice of a target space $X$, a choice of a holomorphic function $W \colon X \fl \Cx$ called a superpotential. In particular, non-trivial Landau-Ginzburg models require a non-compact target space $X$. For a smooth affine variety $X=\uSpec A$, a simple description of the category of D-branes in Landau-Ginzburg models has been proposed by M. Kontsevich and derived from physical considerations in \cite{KL02}. It turns out that the category of D-branes is equivalent to the category $\uMF(W)$ of matrix factorizations of $W$. 

For non-affine $X$ the following construction was proposed \cite{O105}. Suppose that we are given a Landau-Ginzburg superpotential $W\colon X \fl \Cx$ with a single critical value at $0 \in \Cx$. Let $X_0$ denote the fiber of $W$ over $0$. Consider the bounded derived category of coherent sheaves on $X_0$. A perfect complex is an object of $\Dd(X_0)$ which is quasi-isomorphic to a bounded complex of locally free sheaves. One can define a triangulated category of singularities $\Dd_{\Sg}(X_0)$ as the quotient of $\Dd(X_0)$ by the full subcategory of perfect complexes $\uPerf(X_0)$. If $X_0$ were non-singular, the quotient would be trivial, since in that case any object in $\Dd(X_0)$ would have a finite locally free resolution. Therefore $\Dd_{\Sg}(X_0)$ depends only on the singular points of $X_0$. The main result of \cite{O105} is that the category of matrix factorizations $\uMF(W)$ for a smooth affine $X=\uSpec A$ is equivalent to $\Dd_{\Sg}(X_0)$. Thus for non-affine $X$ the category $\Dd_{\Sg}(X_0)$ can be considered as a definition of the category of D-branes. 

One may also consider Landau-Ginzburg models on orbifolds. Such models are particularly important because they provide an alternative description of certain Calabi-Yau sigma models. In the affine case D-branes are described by the category $\uMF^G(W)$ of $G$-equivariant matrix factorizations, cf.~\cite{ADD04,ADDF04} and Section~\ref{sec5} of this paper. In general, one may consider a full subcategory of perfect complexes $\uPerf([X_0/G])$, which is formed by bounded complexes of locally free sheaves in $\Dd([X_0/G])\cong \Dd^G(X_0)$, and also the quotient category $\Dd^G_{\Sg}(X_0)=\Dd^G(X_0)/\uPerf([X_0/G])$. In Section~\ref{sec6} we show that the category of $G$-equivariant matrix factorizations $\uMF^G(W)$ for a smooth affine $X=\uSpec A$ is equivalent to $\Dd^G_{\Sg}(X_0)$. 

Let us assert our version of the McKay correspondence for Landau-Ginzburg models. Consider the Landau-Ginzburg model on the affine space $M=\Cx^n$ with polynomial superpotential $f\colon M \fl \Cx$ and its orbifold with respect to the action of some finite subgroup $G$ of $\uSL(n,\Cx)$. Let $\tau\colon Y \fl  M/G$ be a crepant resolution and consider the Landau-Ginzburg model $(Y,g)$, where $g$ is the pullback of $f$ to $Y$. We expect the following to hold.

\begin{assertion}\label{assertion1.2}
The category of D-branes in the Landau-Ginzburg model $(Y,g)$ is equivalent to the category of D-branes in the Landau-Ginzburg orbifold $(M,f)$.  
\end{assertion}

In this paper we prove a special case of this assertion. The main result is the following. Consider the Landau-Ginzburg orbifold defined by $(M,f)$, where the superpotential $f$ is a regular $G$-invariant function with an isolated critical point at the origin and $G$ is a finite subgroup of $\uSL(n,\Cx)$ which acts on $M=\Cx^n$ freely outside the origin. Assuming favorable circumstances, a crepant resolution is given by the irreducible component $Y \subset G\text{-}\!\uHilb(M)$ dominating $X=M/G$. Then the category of singularities $\Dd_{\Sg}(Y_0)$ of the fiber $Y_0$ is equivalent to the $G$-equivariant category of singularities $\Dd_{\Sg}^G(M_0)$ of the fiber $M_0$. Bearing in mind that the categories of singularities are equivalent to the categories of D-branes we obtain the connection between D-branes mentioned above.   

To finish this introduction we make some remarks of a more philosophical nature. Noncommutative geometry, as propagated by M.~Kontsevich in \cite{Kont98} is based on the idea that to do geometry you really don't need a space, all you need is a category of sheaves on this would-be space. A noncommutative space $X$ is a small triangulated $\Cx$-linear category $\Cs_X$ which is Karoubi closed and enriched over complexes of $\Cx$-vector spaces (this notion is explained in detail in \cite{CDHPS07}). If $X$ is a smooth scheme of finite type, then $X$ can be considered as a noncommutative space with $\Cs_X=\Dd(X)$. Any Landau-Ginzburg model $(X,W)$ is also a noncommutative space with $\Cs_{(X,W)}=\Dd_{\Sg}(X_0)$. We see that the physical meaning of noncommutative space is to replace the space by the category of D-branes. If we return to the McKay correspondence, then we deduce that the noncommutative space $Y$ is isomorphic to the noncommutative space $A=\Cx[x,y,z]* G$. This leads naturally to a generalized notion of McKay correspondence as an isomorphism of noncommutative spaces. Note that this fits well with M. Reid's Principle \ref{ReidPrinciple}, where the word  \textquoteleft geometry\textquoteright~was left deliberately vague. We can restate assertion~\ref{assertion1.2} by saying that the Landau-Ginzburg model $(Y,g)$ and the Landau-Ginzburg orbifold $(M,f)$ are isomorphic as noncommutative spaces.  

\begin{note}
After this paper was posted on the arXiv, I have learned that similar results were obtained by S.~Mehrotra in his PhD dissertation \cite{Meh05}. In the situation described above, he has shown that the $G$-equivariant category of singularities $\Dd_{\Sg}^G(M_0)$ embeds fully and faithfully into the category of singularities $\Dd_{\Sg}(Y_0)$. However, Mehrotra approach is different to ours in that it does not use the techniques of \cite{CH02} in the context of the generalized McKay correspondence. Our proof uses in an essential way these techniques. It is a natural question to try and understand to what extent the result really depends on the derived McKay correspondence, but not a question we explore in this paper.  
\end{note}

\begin{ack}
{\footnotesize I'd like to thank K.~Hori, whose lecture in Vienna led me to the problem. Conversation with many people were very helpful, the incomplete list includes D.~Orlov, A.~King, P.~Horja, A.~Craw, E.~Looijenga, D.~Siersma. I am very grateful to all these people. Many thanks also go to J.~Stienstra for his suggestions and advice. I have also benefitted a lot from e-mail correspondence with D.~H.~Ruip\'{e}rez and M.~Herbst. Finally, I am grateful to the anonymous referee, whose many comments helped me to correct the content and improve the exposition of the manuscript.}
\end{ack}

\section{The physical argument}\label{phys-arg}
We begin with some heuristic physical discussion aimed at justifying assertion \ref{assertion1.2}. The set-up is the so-called gauged linear sigma model.

The gauged linear sigma model is a very useful model which in an appropriate sense \textquoteleft interpolates\textquoteright~between nonlinear sigma models on Calabi-Yau manifolds and Landau-Ginzburg orbifolds. Such a model is determined by a \textquotedblleft radial\textquotedblright~parameter $r$.

Here are some of the basic ideas concerning gauged linear sigma models. We will just indicate enough details to see the parameter $r$ appearing. Let us consider the $\uU(1)$ gauge theory with $n$ chiral matter superfields $X_1,\dots,X_n$ of charge $1$, and one chiral superfield $P$ of charge $-n$. We also consider a twisted chiral superfield $\Sigma$ with values in the complexification of the adjoint bundle over $2|4$-superspace. Write each of these superfields in components
\begin{align*}
X_i &=x_i+\theta(\cdots)+\cdots \\
P &=p+\theta(\cdots)+\cdots \\
\Sigma &= \sigma +\theta(\cdots)+\cdots
\end{align*}
The bosonic potential is a function $V=V(x,p,\sigma)$ of the bosonic components of these superfields. It has the form
$$
V=\frac{1}{2 e^2}D^2+|\sigma|^2 \Big( \sum_{i=1}^n |x_i|^2 + n^2 |p|^2  \Big).
$$
The \textquotedblleft D-term\textquotedblright~is equal to
$$
D=\sum_{i=1}^n |x_i|^2 -n |p|^2-r.
$$
This is actually a familiar function mathematically; it is the moment map generating the
$\uU(1)$-action on the flat K\"{a}hler manifold $Z=\Cx^{n+1}$ with coordinates $x_1,\dots,x_n$ and $p$.

The moduli space of classical vacua \textendash that is, the special field configurations of minimal
energy\textendash~for this theory is
$$
\mathscr{M}_{\mathrm{vac}}=V^{-1}(0)/\uU(1).
$$
The quotient by $\uU(1)$ comes from the gauge symmetry. So we need to set $V=0$ and divide by $\mathrm{U}(1)$. Thanks to the form of the potential, this requires that $D=0$, and either $\sigma=0$ or $ \sum_i |x_i|^2 + n^2 |p|^2=0$.  Now,
setting $D=0$ and dividing by $\mathrm{U}(1)$ is the familiar mathematical operation of symplectic reduction, in which $D=0$
defines a level set for the moment map of the $\mathrm{U}(1)$-action (with the choice of $r$ specifying the level). There is
another mathematical interpretation of this process, as a quotient in the sense of GIT: we complexify the group $\mathrm{U}(1)$ to
$\Cx^{\times}$ and consider the action of $\Cx^{\times}$ on $Z=\Cx^{n+1}$ with the same weights as before (the $x_i$'s have weight $1$ and $p$ has
weight $-n$). 

It turns out that there are two possible GIT quotients depending upon the sign of $r$. For $r>0$, $D=0$ implies that not all $x_i$ can vanish and thus $\sigma$ must be zero. The variable $p$ is free as long as the condition $D=0$ is satisfied. Owing to these, the quotient can be interpreted as the total space $Y=\utot(\Os_{\Px^{n-1}}(-n))$ of the line bundle
$\Os_{\Px^{n-1}}(-n)$ ($p$ serves as a fiber coordinate). For $r<0$, vanishing of the D-term requires that $p \neq 0$. We can therefore use the $\Cx^{\times}$-action on $(x_i,p)$ to set $p=1$. This leaves a residual invariance under the subgroup $G=\Zx_n$ on $\uU(1)$ (because $p$ has charge $-n$). Thus, the quotient is $\Cx^n/G$. This will therefore be what is known as an {\it orbifold} theory.

Let us note that $r$ determines the \textquotedblleft size\textquotedblright~of the non-compact Calabi-Yau manifold $Y$. In this sense, the variable $r$ can be thought of as determining the K\"{a}hler modulus of the theory. Geometrically, taking $r \fl 0$ corresponds to blowing-down the $\Px^{n-1}$ at the base of the line bundle $\Os_{\Px^{n-1}}(-n)$ and the geometry becomes isomorphic to $\Cx^n/G$. 

The real K\"{a}hler modulus $r$ is complexified by the $\theta$-angle of the gauged linear sigma model (which becomes the B-field in string theory) through the combination $\frac{\theta}{2 \pi}+\ui r$, and the complexified K\"{a}hler moduli space has two phases. When $r \gg 0$ the infrared fixed point of the gauged linear sigma model is a nonlinear sigma model on the target space $Y$ and this is called the Calabi-Yau phase. The phase $r \ll 0$ corresponds formally to an analytic continuation to negative K\"{a}hler class. For $\Os_{\Px^{n-1}}(-n)$ this means \textquotedblleft negative size\textquotedblright~of the $\Px^{n-1}$, i.e., we pass to the blow-down phase where the $\Px^{n-1}$ has been collapsed to a point, and the target is $\Cx^n/G$. The singularity at $r=0$ can be avoided by turning on a non-zero $\theta$-angle.

We are particularly interested in trying to understand D-branes (in particular, D-branes with B-type boundary conditions) in gauged linear sigma models with boundary. In the Calabi-Yau phase the category of D-branes is $\Dd(Y)$, the derived category of coherent sheaves on $Y$. In the orbifold phase, this should be replaced by the derived category $\Dd^G(\Cx^n)$ of $G$-equivariant sheaves on $\Cx^n$. We can try to use the boundary gauged linear sigma model as a tool to \textquotedblleft flow\textquotedblright~the category $\Dd^G(\Cx^n)$ to the category $\Dd(Y)$, thus realizing the equivalence of the two categories by means of a physical system. Thus D-branes give a completely natural explanation of the McKay correspondence in terms of the interpolation between small and large \textquotedblleft volume\textquotedblright~phase of a gauged linear sigma model with boundary. 

Now it is time to supplement the gauged linear sigma model by a superpotential $W\colon Z \fl \Cx$. It must be a holomorphic function on $Z=\Cx^{n+1}$. We are chiefly interested in superpotentials of the form $W=pf(x_1,\dots,x_n)$, where $f$ is a general homogeneus polynomial of degree $d$. The potential energy for this linear sigma model is  
$$
V=\frac{1}{2 e^2}D^2+|f|^2+|p|^2|\ud f|^2+|\sigma|^2 \Big( \sum_{i=1}^n |x_i|^2 + n^2 |p|^2  \Big).
$$       
Let us restrict attention to polynomials that are {\it transverse}, meaning that the equations $f=\ud f=0$ have no simultaneous solutions except at the origin. This implies that the hypersurface $S$ of
$\Px^{n-1}$ defined by $f=0$ is a smooth complex manifold. Moreover, if $d=n$ then $S$ is a Calabi-Yau
manifold. We will assume this in the sequel.

Let us analyse the spectrum of the classical theory. As before, the structure of the moduli space of classical vacua is different for $r>0$ and $r<0$, and we will treat these two cases separately.

First, let us take $r>0$. In this case, $D=0$ requires at least one $x_i$ to be nonzero, forcing $\sigma$ to vanish. If we assume $p \neq 0$, the equations $f=\ud f=0$ with the transversality condition imply that all $x_i$ must vanish. However, this is inconsistent with $D=0$. Thus $p$ must be zero. Our equations for classical vacua become $p=0$, $\sum_{i} |x_i|^2=r$, and $f=0$, and we must divide by the action of the gauge group $\uU(1)$. This gives the hypersurface $S$ defined by the equation $f=0$ in $\Px^{n-1}$, with K\"{a}hler modulus $r$. Thus, classically our theory can be described as a nonlinear sigma model whose target space is this hypersurface $S$.

Let us move to the case $r<0$. The space of classical vacua satisfies $x_i=0$ and $n|p|^2=-r$. We can use a gauge transformation to fix $p=\sqrt{-r/n}$, leaving a residual gauge invariance of $G=\Zx_n$. The local description of the theory is this: for $r \ll 0$, the field $P$ has a large mass and can be integrated out, leaving an effective theory of $n$ massless chiral superfields $X_1,\dots,X_n$ with an effective interaction
$$
W_{\mathrm{eff}}=\mathrm{const} \cdot f(x_1,\dots,x_n).
$$
Such a theory of $n$ massless fields with a polynomial interaction is called a Landau-Ginzburg model. We should notice, however, that the Landau-Ginzburg model is not an ordinary one, but a $G$-gauge theory. Physical fields must be invariant under the $G$-action, and the configuration must be single-valued only up to the $G$-action. Such a gauge theory is usually called a Landau-Ginzburg orbifold. 

In this way, the gauged linear sigma model interpolates between the Landau-Ginzburg orbifold and the Calabi-Yau nonlinear sigma model. These two regions can be considered as a sort of analytic continuation of each other. 

In both these theories we know how to describe topological D-branes. In the Calabi-Yau phase the D-brane category is the derived category $\Dd(S)$ of coherent sheaves on $S$. In the Landau-Ginzburg phase, D-branes are realized as $G$-equivariant matrix factorizations of $f$.  Using the gauged linear sigma model realization, the previous discussion naturally leads to the statement that there should be an equivalence of categories $\Dd(S) \cong \uMF^G(f)$, where $\uMF^G(f)$ is the category of $G$-equivariant matrix factorizations of $f$. 

Now, we can consider $Y=\utot(\Os_{\Px^{n-1}}(-n))$ as a Landau-Ginzburg model with superpotential $g$ given by the pullback of $f$ to $Y$. As mentioned in the introduction, in this case the category of D-branes is defined as the category of singularities $\Dd_{\Sg}(Y_0)$, where $Y_0$ is the fiber of $g$ over $0$. 

On the other hand, we can describe $Y$ as a GIT quotient of an affine space $Z=\Cx^{n+1}$ by the linear action of $\Cx^{\times}$. The underlying superpotential $W=p f(x_1,\dots,x_n)$ on $Z=\Cx^{n+1}$ descends to a holomorphic function on $Y$ that coincides with $g$. In the presence of a $\Cx^{\times}$-action one can also consider the category $\uMF^{\gr}(W)$ of graded matrix factorizations of $W$. We can think of the latter as being the category of D-branes in the gauged linear sigma model.

Now we reach the crucial step. One of the main outcomes of \cite{HHP08}, is that the categories of D-branes in the Calabi-Yau and Landau-Ginzburg phases are both quotients of $\uMF^{\gr}(W)$. However, at $r >0$ and at  \textquotedblleft intermidiate energy scale\textquotedblright~one could always choose the description as the Landau-Ginzburg model with superpotential $g$ over $Y$.  This superpotential gives masses to the field $P$ and to the \textquotedblleft transverse modes\textquotedblright~to the hypersurface $S$. At \textquotedblleft lower energies\textquotedblright, it is more appropriate to integrate them out, and we have the nonlinear sigma model on $S$.

In the light of all this we can expect that the categories of D-branes $\Dd_{\Sg}(Y_0)$ and $\uMF^G(f)$ are also equivalent.
Now, our Theorem~\ref{branes-singularities} gives an equivalence between the category of D-branes $\uMF^G(f)$ and the $G$-equivariant category of singularities $\Dd_{\Sg}^G(M_0)$, where $M_0$ is the fiber of $f$ over $0$. So, we arrive at the statement that the category $\Dd_{\Sg}(Y_0)$ should be equivalent to the category $\Dd_{\Sg}^G(M_0)$. This equivalence allows us to compare the category of D-branes on the Landau-Ginzburg model $(Y,g)$ with the category of D-branes in the Landau-Ginzburg orbifold $(M,f)$. Given this simple observation, it is natural to think that the correspondence between D-branes in the two theories is given by a McKay correspondence.  

 \section{Localization in triangulated categories}\label{sec1}
In this section we will review the definition of localization of triangulated categories. The reader is referred to \cite{GM96}, for example, for a more complete discussion.

Recall that a triangulated category $\mathscr{D}$ is an additive category equipped with the additional data:
\begin{enumerate}
\item[(a)]an additive autoequivalence $T\colon  \mathscr{D} \fl \mathscr{D}$, which is called a translation functor,
\item[(b)]a class of exact (or distinguished) triangles
$$
\xymatrix@C-1mm{X \ar[r]^-{u} & Y \ar[r]^-{v} & Z \ar[r]^-{w} & TX.}
$$
\end{enumerate}
This data must satisfy a certain set of axioms (see \cite{GM96}, also \cite{H66}).

An additive functor $F\colon  \mathscr{D} \fl \mathscr{D}'$ between two triangulated categories $\mathscr{D}$ and $\mathscr{D}'$ is called {\it exact} if it commutes with the translation functors, i.e.~there is a natural isomorphism $FT \cong TF$, and it
sends exact triangles to exact triangles, i.e.~any exact triangle $X \fl Y \fl Z
\fl TX$ in $\mathscr{D}$ is mapped to an exact triangle
$$
\xymatrix@C-1mm{FX \ar[r] & FY \ar[r] & FZ\ar[r] & FTX}
$$
in $\mathscr{D}'$, where $FTX$ is identified with $TFX$ via the natural isomorphism of $FT$ and $TF$.

A full additive subcategory $\mathscr{N} \subset \mathscr{D}$ is said to be a full triangulated subcategory, if the
following condition holds: it is closed with respect to the translation functor in $\mathscr{D}$ and if it contains
any two objects of an exact triangle in $\mathscr{D}$ then it contains the third object of this triangle as well.

With any pair $\mathscr{N} \subset \mathscr{D}$, where $\mathscr{N}$ is a full triangulated subcategory in a
triangulated category $\mathscr{D}$, we can associate the {\it quotient} $\mathscr{D}/\mathscr{N}$. To construct
it denote by $\Sigma$ a class of morphisms $s$ in $\mathscr{D}$ fitting into an exact triangle
$$
\xymatrix@C-1mm{X \ar[r]^-{s} &  Y \ar[r] & N \ar[r] & TX}
$$
with $N \in \mathscr{N}$. It is not hard to see that $\Sigma$ is a multiplicative system. We then define the quotient
$\mathscr{D}/\mathscr{N}$ as the localization $\mathscr{D}[\Sigma^{-1}]$ and observe that it is a triangulated category. The
translation functor on $\mathscr{D}/\mathscr{N}$ is induced from the translation functor in the category
$\mathscr{D}$, and the exact triangles in $\mathscr{D}/\mathscr{N}$ are triangles isomorphic to the images of exact
triangles.

The category $\mathscr{D}/\mathscr{N}$ has the following explicit description. The objects of
$\mathscr{D}/\mathscr{N}$ are the objects of $\mathscr{D}$.  The
morphisms from $X$ to $Y$ are equivalence classes of diagrams $(s,f)$ in $\mathscr{D}$ of the form
$$
\xymatrix@C-1mm{X  & Y' \ar[r]^-{f} \ar[l]_-{s}& Y } \quad \text{with $s \in \Sigma$,}
$$ 
where two diagrams $(s,f)$ and $(t,g)$ are equivalent if they fit into a commutative diagram
$$
\xymatrix{& Y' \ar[rd]^-{f} \ar[ld]_-{s}& \\
X   & Y''' \ar[r]^-{h} \ar[l]_-{r} \ar[u] \ar[d] & Y   \\
& Y''  \ar[lu]^-{t} \ar[ru]_-{g}& }
$$
with $r \in \Sigma$.

The quotient functor $Q\colon \mathscr{D} \fl \mathscr{D}/\mathscr{N}$ annihilates $\mathscr{N}$. Moreover, any exact
functor $F\colon  \mathscr{D} \fl \mathscr{D}'$ between triangulated categories, for which $F(X) \cong 0$ when $X \in
\mathscr{N}$, factors uniquely through $Q$. This implies the following result which will be useful later.

\begin{lem}\label{lem3.1}
Let $\mathscr{N}$ and $\mathscr{N}'$ be full triangulated subcategories of triangulated categories $\mathscr{D}$ and $\mathscr{D}'$, respectively. Let $F\colon \mathscr{D} \fl \mathscr{D}'$ and $G\colon \mathscr{D}' \fl \mathscr{D}$ be an adjoint pair of exact functors such that $F(\mathscr{N}) \subset \mathscr{N}'$ and $G(\mathscr{N}') \subset \mathscr{N}$. Then they induce functors
$$
\overline{F}\colon  \mathscr{D}/\mathscr{N} \lfl \mathscr{D}'/\mathscr{N}' \quad \text{and} \quad \overline{G}\colon  \mathscr{D}'/\mathscr{N}' \lfl \mathscr{D}/\mathscr{N}
$$
which are adjoint as well. Moreover, if the functor $F\colon  \mathscr{D} \fl \mathscr{D}'$ is fully faithful, then the functor $\overline{F}\colon  \mathscr{D}/\mathscr{N} \fl \mathscr{D}'/\mathscr{N}'$ is also fully faithful.
\end{lem}

\section{Triangulated categories of singularities}\label{sec3}
In this section we give the definition and basic properties of triangulated categories of singularities. We refer to
Orlov's papers \cite{O105} and \cite{O205} for all the proofs of the assertions below.

We are mainly interested in triangulated categories and their quotient by triangulated subcategories which are
coming from algebraic geometry. Let $X$ be a separated Noetherian scheme of finite Krull dimension over $\Cx$ such that the category of coherent sheaves $\uCoh(X)$ has enough locally free sheaves. For future reference we denote the category of quasi-coherent sheaves on $X$ by $\uQcoh(X)$. 

Denote by $\Dd(X)$ the bounded derived category of coherent sheaves on $X$. The objects of the category $\Dd(X)$
which are isomorphic to bounded complexes of locally free sheaves on $X$ form a full triangulated subcategory. It is
called the subcategory of {\it perfect complexes} and is denoted by $\uPerf(X)$.\footnote{Actually, a perfect
complex is defined as a complex of $\Os_X$-modules locally quasi-isomorphic to a bounded complex of locally free
sheaves of finite type. But under our assumption on the scheme any such complex is quasi-isomorphic to a bounded
complex of locally free sheaves of finite type (see \cite{TT90}).}

\begin{dfn}
Define the triangulated category of singularities $\Dd_{\Sg}(X)$ of $X$ as the quotient category
$\Dd(X)/\uPerf(X)$.
\end{dfn}

It is known that if our scheme $X$ is regular then the subcategory of perfect complexes $\uPerf(X)$ coincides with
the whole bounded derived category of coherent sheaves. In this case the triangulated category of singularities
$\Dd_{\Sg}(X)$ is trivial. Thus $\Dd_{\Sg}(X)$ is only sensitive to singularities of $X$.

Let $f\colon X \fl Y$ be a morphism of finite $\mathrm{Tor}$-dimension (for example a flat morphism or a regular closed embedding). It
defines the inverse image functor $\Ld f^*\colon  \Dd(Y) \fl \Dd(X)$. It is clear that the functor $\Ld f^*$ sends perfect complexes on
$Y$ to perfect complexes on $X$. Therefore, the functor $\Ld f^*$ induces an exact functor $\Ld
\overline{f}^*\colon \Dd_{\Sg}(Y) \fl \Dd_{\Sg}(X)$.

Suppose, in addition, that the morphism $f\colon X \fl Y$ is proper and locally of finite type. Then the direct image functor $\Rd
f_*\colon \Dd(X) \fl \Dd(Y)$ takes perfect complexes on $X$ to perfect complexes on $Y$ (see \cite{TT90}). Hence it determines a
functor $\Rd \overline{f}_*\colon \Dd_{\Sg}(X) \fl \Dd_{\Sg}(Y)$ which is right adjoint to $\Ld
\overline{f}^*$. We should remark, however, that all the specific morphisms we consider are non-proper.

A fundamental property of triangulated categories of singularities is a property of locality. Here is a precise statement. 

\begin{prop}\label{locality}
Let $X$ be as above and let $j\colon  U \fl X$ be an embedding of an open subscheme such that $\uSing(X) \subset U$. Then the
functor $\overline{j}^*\colon \Dd_{\Sg}(X) \fl \Dd_{\Sg}(U)$ is an equivalence of triangulated categories.
\end{prop} 

Triangulated categories of singularities of $X$ have additional good properties in case the scheme is Gorenstein. Recall that
a local Noetherian ring $A$ is called Gorenstein if $A$ as module over itself has a finite injective resolution. It can be shown
that if $A$ is Gorenstein then $A$ has finite injective dimension and the natural map
$$
M \lfl \Rd \uHom\spdot_A(\Rd\uHom\spdot_A(M,A),A)
$$ 
is an isomorphism for any finitely generated $A$-module $M$ and, as a consequence, for any object from $\Dd(\uSpec A)$. A
scheme $X$ is Gorenstein if all of its local rings are Gorenstein local rings. If $X$ is Gorenstein and has finite dimension,
then $\Os_X$ is a dualizing complex for $X$, i.e.~it has finite injective dimension as a quasi-coherent sheaf and the natural map
$$
\Es \lfl \Rd \Hs om\spdot_{X}(\Rd \Hs om\spdot_{X}(\Es,\Os_X),\Os_X)
$$ 
is an isomorphism for any coherent sheaf $\Es$. In particular, there is an integer $n_0$ such that $\Es xt_X^i(\Es,\Os_X)=0$ for
each quasi-coherent sheaf $\Es$ and all $i>n_0$.

The following gives a useful description of the morphism spaces in triangulated categories of singularities.

\begin{prop}\label{prop2.3}
Let $X$ be as above and Gorenstein. Let $\Es$ and $\Fs$ be coherent sheaves such that $\Es xt_X^i(\Es,\Os_X)=0$ for all $i>0$.
Fix $n$ such that $\Es xt_X^i(\Ss,\Fs)=0$ for $i >n$ and for any locally free sheaf $\Ss$. Then
$$
\uHom_{\Dd_{\Sg}(X)}(\Es,\Fs[n])\cong \uExt_X^n(\Es,\Fs)/\Rs
$$
where $\Rs$ is the subspace of elements factoring through locally free, i.e. $e \in \Rs$ if and only if $e=\alpha \beta$ with
$\alpha\colon  \Es \fl \Ss$ and $\beta \in \uExt_X^n(\Ss,\Fs)$ where $\Ss$ is locally free.
\end{prop}

\section{Triangulated categories of matrix factorizations}\label{sec4}
In this section we introduce the category of matrix factorizations and give some of its basic properties. The origin of this category goes back to the work of D.~Eisenbud \cite{Eis80} in the context of so-called maximal Cohen-Macaulay modules over local rings of hypersurface singularities.

As proposed by M.~Kontsevich (see
also \cite{KL02}) the category of D-branes associated to a Landau-Ginzburg model can be characterized in terms of matrix factorizations. For us, a Landau-Ginzburg model is simply a pair $(X,W)$, where $X$ is a smooth variety (or regular scheme), and $W\colon X \fl \Cx$ is
a regular function on $X$ called the {\it superpotential}. To keep things simple, we will assume throughout that $W$ has a single critical value at the origin $0 \in \Cx$. To this data one can associate two categories: an exact category
$\uPair(W)$ and a triangulated category $\uMF(W)$. We give the construction of these categories under the condition
that $X$ is affine.

Let $A$ be a commutative algebra over $\Cx$. Then one can regard $A$ as the algebra of functions on an affine scheme
$X=\uSpec A$. Denote by $ \uMod A$ the category of all right modules over $A$. It is a well-known fact that the global section functor
$$
H^0\colon \uQcoh(X) \lfl \uMod A,
$$
is an equivalence with inverse denoted by $\widetilde{(-)}$. It is also well-known that this functor restrict to an equivalence
$$
H^0\colon \uCoh(X) \lfl \umod A,
$$
where $\umod A$ is the category of finitely generated right modules over $A$. Note that under this equivalence locally free sheaves are the same as projective modules.

For a non-zero element $W \in A$, a {\it matrix factorization} of $W$ is an ordered pair
$$
\overline{P}=\big( \xymatrix{P_0 \ar@<0.4ex>[r]^-{p_0} & P_1
 \ar@<0.4ex>[l]^-{p_1}}  \big)
$$ 
where $P_0$, $P_1$ are finitely generated projective $A$-modules and $p_0$, $p_1$ are $A$-ho\-mo\-mor\-phisms such that $p_1 p_0=W
\cdot \mathrm{id}_{P_0}$ and $p_0p_1=W \cdot \mathrm{id}_{P_1}$. Since $p_0p_1$ and $p_1p_0$ are $W$ times the identities, where
$W$ is a non-zero element of $A$, the rank of $P_0$ coincides with that of $P_1$. We call the rank the size of the matrix
factorization.

The above construction can be reformulated in terms of $\Zx_2$-graded $A$-modules as follows. A $\Zx_2$-graded $A$-module $P=P_0
\oplus P_1$ can be thought of as an ordinary $A$-module $P$ equipped with a $\Cx$-linear involution $\tau\colon  P \fl P$,
$\tau^2=\mathrm{id}$. The homogeneous parts $P_0$ and $P_1$ are the eigenspaces of $\tau$ corresponding to the eigenvalues $1$
and $-1$ respectively. A pair $\overline{P}$ can be similarly thought of as a triple $(P,\tau,D_P)$ where $D_P\colon P \fl P$ is an odd
$A$-homomorphism satisfying $D_P^2=W \cdot \mathrm{id}_P$. Given two matrix factorizations $\overline{P}=(P,\tau,D_P)$ and
$\overline{Q}=(Q,\sigma,D_Q)$ the $A$-module $\bHom(\overline{P},\overline{Q})$ form a $\Zx_2$-graded complex
$$
\bHom(\overline{P},\overline{Q})=\bHom(\overline{P},\overline{Q})_0\oplus
\bHom(\overline{P},\overline{Q})_1
$$
where
\begin{align*}
\bHom(\overline{P},\overline{Q})_0&=\uHom_A(P_0,Q_0)\oplus \uHom_A(P_1,Q_1), \\
\bHom(\overline{P},\overline{Q})_1&=\uHom_A(P_0,Q_1)\oplus \uHom_A(P_1,Q_0),
\end{align*}
and with differential $D$ acting on homogeneous elements of degree $k$ as
$$
D \phi=D_{Q} \cdot \phi - (-1)^k \phi \cdot D_P.
$$

The set of objects of the categories $\uPair(W)$ and $\uMF(W)$ is given by the set of matrix factorizations of $W$.
The space of morphisms $\uHom_{\uPair(W)}(\overline{P},\overline{Q})$ in the category $\uPair(W)$ is the
space of homogeneous morphisms of degree $0$ which commute with the differential $D$. The space of morphisms in the category
$\uMF(W)$ is the space of morphisms in $\uPair(W)$ modulo null-homotopic morphisms, i.e.
\begin{align*}
\uHom_{\uPair(W)}(\overline{P},\overline{Q})&=Z^0(\bHom(\overline{P},\overline{Q})),\\
\uHom_{\uMF(W)}(\overline{P},\overline{Q})&=H^0(\bHom(\overline{P},\overline{Q})).
\end{align*}
Thus a morphism $\phi\colon  \overline{P}\fl \overline{Q}$ in the category $\uPair(W)$ is a pair of morphisms $\phi_0\colon P_0 \fl
Q_0$ and $\phi_1\colon P_1 \fl Q_1$ such that $\phi_1 p_0=q_0 \phi_0$ and $q_1 \phi_1=\phi_0 p_1$. The morphism $\phi$ is
null-homotopic if there are two morphisms $t_0\colon P_0 \fl Q_1$ and $t_1\colon P_1 \fl Q_0$ such that $\phi_1=q_0 t_1+t_0 p_1$ and
$\phi_0=t_1 p_0 + q_1 t_0$.

It is clear that the category $\uPair(W)$ is an exact category with respect to componentwise monomorphisms and
epimorphisms (see definition in \cite{Qui73}).

The category $\uMF(W)$ can be endowed with a natural structure of a triangulated category. To determine it we have to
define a translation functor $[1]$ and a class of exact triangles.

The translation functor can be defined as a functor that takes $\overline{P}$ to the object
\begin{equation}\label{shift}
\overline{P}[1]=\big( \xymatrix{P_1 \ar@<0.4ex>[r]^-{-p_1} & P_0
 \ar@<0.4ex>[l]^-{-p_0}}  \big)
\end{equation}
i.e.~it changes the order of the modules and signs of the morphisms, and takes a morphism $\phi=(\phi_0,\phi_1)$ to the morphism
$\phi[1]=(\phi_1,\phi_0)$. We see that the functor $[2]$ is the identity functor. 

For any morphism $\phi\colon   \overline{P}\fl \overline{Q}$ from the category $\uPair(W)$ we define a mapping cone $C(\phi)$ as
an object
\begin{equation}\label{cone}
C(\phi)=\big( \xymatrix{Q_0 \oplus P_1 \ar@<0.4ex>[r]^-{c_0} & Q_1 \oplus P_0
 \ar@<0.4ex>[l]^-{c_1}}  \big)
\end{equation}
such that 
$$
c_0=\left( \begin{array}{cc} q_0 & \phi_1 \\
                                               0 & -p_1 \end{array}  \right), \quad c_1=\left( \begin{array}{cc}
      q_1 & \phi_0 \\
                                                0 & -p_0 \end{array}  \right).
$$
There are maps $\psi\colon  \overline{Q} \fl C(\phi)$, $\psi=(\mathrm{id},0)$ and $\xi\colon C(\phi) \fl \overline{P}[1]$,
$\xi=(0,\mathrm{id})$.

Now we define a standard triangle in the category $\uMF(W)$ as a triangle of the form
$$
\xymatrix@C-1mm{\overline{P} \ar[r]^-{\phi} & \overline{Q}  \ar[r]^-{\psi} & C(\phi)   \ar[r]^-{\xi} & \overline{P}[1]}
$$
for some $\phi \in \uHom_{\uPair(W)}(\overline{P},\overline{Q})$. A triangle $\overline{P} \fl \overline{Q} \fl
\overline{R} \fl \overline{P}[1]$ in $\uMF(W)$ will be called an exact triangle if it is isomorphic to a standard one.

As a consequence we get the following.

\begin{prop}
The category $\uMF(W)$ endowed with the translation functor $[1]$ and the above class of exact triangles becomes a
triangulated category.
\end{prop}

The proof is the same as the analogous result for a usual homotopic category (see, for example \cite{GM96}).

\begin{dfn}
The category $\uMF(W)$ constructed above is called the triangulated category of matrix factorizations for the pair
$(X=\uSpec A, W)$.
\end{dfn}

Denote by $X_0$ the fiber of $W\colon X \fl \Cx$ over the point $0$. With any matrix factorization $\overline{P}$ we can associate a short exact sequence
$$
\xymatrix@C-1mm{0 \ar[r] & P_1 \ar[r]^-{p_1} & P_0 \ar[r] & \ucoker p_1 \ar[r] & 0.} 
$$
We can attach to an object $\overline{P}$ the sheaf $\ucoker p_1$. This is a sheaf on $X$. But the multiplication by
$W$ annihilates it. Hence, we can consider $\ucoker p_1$ as a sheaf on $X_0$. Any morphism
$\phi\colon \overline{P}\fl\overline{Q}$ in $\uPair(W)$ gives a morphism between cokernels. This way we get a functor
$\uCok\colon  \uPair(W) \fl \uCoh(X_0)$. We have the following result, see \cite[Theorem~3.9]{O105}.

\begin{thm}
There is a functor $F$ which completes the following commutative diagram
$$
\xymatrix{\uPair(W) \ar[r]^-{\uCok} \ar[d]& \uCoh(X_0) \ar[d]\\
          \uMF(W) \ar[r]_-{F} &  \Dd_{\Sg}(X_0).}
$$
Moreover, the functor $F$ is an equivalence of triangulated categories.
\end{thm}

\section{Orbifold categories}\label{sec5}
As is well known, for the Calabi-Yau/Landau-Ginzburg correspondence, one must consider orbifolds of D-branes in a Landau-Ginzburg theory. The definition of triangulated categories of singularities and matrix factorizations can be extended to this situation.

We start by recalling the definition and basic properties of equivariant coherent sheaves. More details can be found in \cite{Pl05}. Let $G$ be a finite group acting on some scheme $X$. A
$G$-equivariant coherent sheaf on $X$ is a coherent sheaf $\Es$ on $X$ together with isomorphisms $\lambda_g^{\Es}\colon  \Es \xrightarrow{\sim} g^* \Es$
for all $g \in G$ subject
to $\lambda_e^{\Es}=\id_{\Es}$ and $\lambda_{gh}^{\Es}=h^*(\lambda_g^{\Es})\lambda_h^{\Es}$. Mumford calls this a $G$-linearization of $\Es$.

If $\Es$ and $\Fs$ are two $G$-equivariant coherent sheaves, then the vector space $\uHom_X(\Es,\Fs)$ becomes a $G$-representation
via $g \cdot \theta=(\lambda_g^{\Fs} )^{-1}g^* \theta \lambda_g^{\Es}$ for $\theta\colon  \Es \fl \Fs$. Let $\uCoh^G(X)$ be
the category whose objects are $G$-equivariant coherent sheaves and whose morphisms are the $G$-invariant sheaf morphisms:
$$
G\text{-}\!\uHom_X(\Es,\Fs)\equiv \uHom_X(\Es,\Fs)^G.
$$ 
This category is abelian. It is not difficult to define the
usual additive functors $\otimes$, $\Hs om$ on this category. Furthermore, if $f\colon X \fl Y$ is a $G$-equivariant map between
$G$-schemes, then one defines in an obvious way the additive functors $f_*\colon \uCoh^G(X) \fl \uCoh^G(Y)$,
$f^*\colon \uCoh^G(Y) \fl \uCoh^G(X)$. For example, if $\Es \in \uCoh^G(X)$, then $f_* \Es$ is canonically a
$G$-equivariant coherent sheaf via $f_* \lambda_g^{\Es}\colon  f_* \Es \xrightarrow{\sim} f_* g^* \Es=g^* f_* \Es$. One now also has the usual
adjunctions and relations among these functors. 

We shall have to deal with the special case where $G$ acts trivially on $X$. Then a $G$-equivariant coherent sheaf $\Es$ is merely given by a group
homomorphism $\lambda^{\Es}\colon G \fl \uAut(\Es)$. As $G$ is finite, this representation decomposes into a direct sum over
the irreducible $G$-representations $\rho_0,\rho_1,\dots,\rho_n$, where we take $\rho_0$ to be the trivial one; i.e.~$\Es \cong
\bigoplus_{i=0}^n \Es_i \otimes_{\Os_X} \widetilde{\rho}_i$ in $\uCoh^G(X)$ with ordinary sheaves $\Es_i \in
\uCoh(X)$. There exists no homomorphisms between sumands corresponding to two different representations, and hence we
obtain two mutually adjoint and exact functors, the latter of which is \textquoteleft taking $G$-invariants\textquoteright:
\begin{align*}
-\otimes \rho_0 &\colon  \uCoh(X) \lfl \uCoh^G(X), \\
[-]^G &\colon  \uCoh^G(X) \lfl \uCoh(X).
\end{align*}

We come back now to the general case. Given two objects $\Es$ and $\Fs$ in $\uCoh^G(X)$, we consider $\uExt_X^i(\Es,\Fs)$ as a $G$-representation in the usual way. Then it is easily seen that
$$
G\text{-}\!\uExt_X^i(\Es,\Fs)=\uExt_X^i(\Es,\Fs)^G.
$$
Denote the bounded derived category of $\uCoh^G(X)$ by $\Dd^G(X)$. We shall refer to $\Dd^G(X)$ as the derived category of $G$-equivariant coherent sheaves on $X$. Using induction on the length of complexes,
the above relation for equivariant $\uExt$ groups translates to
$$
\uHom_{\Dd^G(X)}^i(\Es\spdot,\Fs\spdot)=\uHom_{\Dd(X)}^i(\Es\spdot,\Fs\spdot)^G,
$$
for complexes of $G$-equivariant coherent sheaves $\Es\spdot$ and $\Fs\spdot$ in $\Dd^G(X)$. Note that all facts about $G$-equivariant coherent sheaves also apply to complexes of $G$-equivariant coherent sheaves.

It will be useful for us to look at $\Dd^G(X)$ in another way. Consider the quotient stack $[X/G]$. It is covered by one \'{e}tale chart, given by the projection $X \fl X/G$, or more explicitly, by the fiber diagram
$$
\xymatrix{G \times X \ar[r]^-{p} \ar[d]_-{\sigma}& X \ar[d] \\
            X \ar[r] & X/G.}
$$
Now a sheaf on the stack $[X/G]$ is just a sheaf $\Es$ on the chart $X$ with $p^* \Es \cong \sigma^* \Es$, and the
descend condition translates into the linearization property. Therefore, the abelian categories $\uCoh([X/G])$
and $\uCoh^G(X)$ are equivalent, and consequently they give rise to equivalent derived categories.

A {\it perfect complex} of $G$-equivariant coherent sheaves is an object of $\Dd([X/G])$ which is quasi-isomorphic to a bounded complex of locally free sheaves on $[X/G]$. The perfect complexes of $G$-equivariant coherent sheaves form a full triangulated subcategory $\uPerf([X/G]) \subset \Dd([X/G]) \cong \Dd^G(X)$.

\begin{dfn}
Define the $G$-equivariant category of singularities $\Dd^G_{\Sg}(X)$ of $X$ as the quotient category
$\Dd^G(X)/\uPerf([X/G])$.
\end{dfn}

One can show that the entire discussion we had in Section~\ref{sec3} goes through in the case of $G$-equivariant coherent
sheaves.

It also makes sense to define $G$-equivariant matrix factorizations. Suppose $X=\uSpec A$ is a $G$-scheme. It is natural to define the following abelian category
$\ueMod^G\!\!\text{\textendash} A$. Its objects are $A$-modules $M$ with the property that for every $g \in G$, there is given an
$A$-isomorphism $\lambda_g^M\colon M \fl g^* M$, such that for every $g,h \in G$, we have
$\lambda_{gh}^M=h^*(\lambda_g^M) \lambda_h^M$ and $\lambda_e^M=\id_M$. Note that in this expression $g^*M=g_*^{-1}M$
is just the abelian group $M$ with its $A$-module structure induced by $g^{-1}\colon A \fl A$. A morphism $\phi\colon M \fl N$ is just an
$A$-homomorphism, which should satisfy the property that for all $g \in G$ and  $m \in M$, we have
$\phi(\lambda_g^M(m))=\lambda_g^N(\phi(m))$. This clearly gives rise to an abelian category in a natural way. Likewise, it has
an abelian subcategory determined by the full subcategory of finitely generated $A$-modules, which we will denote by
$\uemod^G\!\!\text{\textendash} A$. Note that if $X$ happens to be a trivial $G$-scheme, we have
$\uemod^G\!\!\text{\textendash} A=\Cx G \text{\textendash} \!\umod A$ (just a category of bimodules). We can now define in an
obvious way a functor
$$
H^0\colon \uQcoh^G(X) \lfl \ueMod^G\!\!\text{\textendash} A,
$$
which is an equivalence with inverse $\widetilde{(-)}$. Moreover this functor restrict to an equivalence
$$
H^0\colon \uCoh^G(X) \lfl \uemod^G\!\!\text{\textendash} A.
$$
Note that these functors are just extensions of the previous ones.

Now assume that there is an action of the group $G$ on the Landau-Ginzburg model $(X=\uSpec A,W)$ such that the
superpotential $W$ is $G$-equivariant. In this case, we can consider two categories: an exact category $\uPair^G(W)$
and a triangulated category $\uMF^G(W)$. Objects of these categories are ordered pairs
$$
\overline{P}=\big( \xymatrix{P_0 \ar@<0.4ex>[r]^-{p_0} & P_1
 \ar@<0.4ex>[l]^-{p_1}}  \big)
$$ 
where $P_0$, $P_1$ are finitely generated projective $G$-$A$-modules and $p_0$, $p_1$ are $G$-equivariant maps such that the
compositions $p_0p_1$ and $p_1p_0$ are the multiplication by the element $W \in A$. A morphism $\phi\colon  \overline{P} \fl
\overline{Q}$ in the category $\uPair^G(W)$ is a pair of $G$-equivariant morphisms $\phi_0\colon P_0 \fl Q_0$ and $\phi_1\colon P_1
\fl Q_1$ such that $\phi_1 p_0=q_0 \phi_0$ and $q_1 \phi_1=\phi_0 p_1$. Morphisms in the category $\uMF^G(W)$ are
classes of $G$-equivariant morphisms in $\uPair^G(W)$ modulo null-homotopic morphisms. The shift functor and the
distinguished triangles can be constructed by imposing equivariance conditions on equations (\ref{shift}) and (\ref{cone}).

\begin{dfn}
The category $\uMF^G(W)$ constructed above is called the triangulated category of $G$-equivariant matrix factorizations for the pair
$(X=\uSpec A, W)$.
\end{dfn}

\section{Categories of matrix factorizations and categories of singularities}\label{sec6}
Our aim now is to describe an equivalence of categories between $\uMF^G(W)$, the category of $G$-equivariant
matrix factorizations and $\Dd^G_{\Sg}(X_0)$, the $G$-equivariant category of singularities. In the
non-equivariant setting, we have seen in Section~\ref{sec4} that $\uMF(W)$ is equivalent to
$\Dd_{\Sg}(X_0)$. The generalization to the equivariant situation is straightforward. Our proofs in this section are modeled on
those in \cite{O105}. 

With any object $\overline{P}$ in $\uPair^G(W)$ we associate the module $\ucoker p_1$ and its free
resolution
$$
\xymatrix@C-1mm{0 \ar[r] & P_1 \ar[r]^-{p_1} & P_0 \ar[r] & \ucoker p_1 \ar[r] & 0.} 
$$
It can be easily checked that $W$ annihilates $\ucoker p_1$. Hence the module $\ucoker p_1$ is
naturally a right $G$-$A$-module. For each object $\overline{P}$ in $\uPair^G(W)$ we define
$\uCok^G(\overline{P})=\ucoker p_1$; this is a $G$-equivariant coherent sheaf on $X_0$. If $\phi\colon \overline{P} \fl
\overline{Q}$ is a morphism in $\uPair^G(W)$ then $\phi$ induces a morphism $\uCok^G(\phi)\colon \ucoker p_1
\fl \ucoker q_1$. This construction defines a functor $\uCok^G\colon \uPair^G(W) \fl
\uCoh^G(X_0)$.

\begin{lem}
The functor $\uCok^G$ is full.
\end{lem}

\begin{proof}
This is essentially the Lemma~3.5 proved in \cite{O105}. We recall its proof for the convenience of readers. Fix two objects $\overline{P}$ and $\overline{Q}$ in $\uPair^G(W)$ and let $f\colon \ucoker p_1
\fl \ucoker q_1$ be a morphism in $\uCoh^G(X_0)$. Since $P_0$ and $P_1$ are projective $f$ can be
extended to a map of exact sequences
$$
\xymatrix{0 \ar[r] & P_1 \ar[r]^-{p_1}\ar[d]_-{\phi_1} & P_0 \ar[r]\ar[d]^-{\phi_0} & \ucoker p_1 \ar[d]^-{f}\ar[r] & 0 \\
0 \ar[r] & Q_1 \ar[r]^-{q_1} & Q_0 \ar[r] & \ucoker q_1 \ar[r] & 0.}
$$
We want to show that $\phi=(\phi_0,\phi_1)$ is a map of pairs. We have that
$$
q_1(\phi_1 p_0-q_0 \phi_0)=\phi_0p_1p_0-q_1q_0\phi_0=\phi_0 W -W \phi_0=0.
$$
Using that $q_1$ is a monomorphism, we get that $\phi_1 p_0=q_0 \phi_0$, which shows that $\phi=(\phi_0,\phi_1)$ is a map of pairs, as required.
\end{proof}

Next we show that the functor $\uCok^G$ induces an exact functor between triangulated categories.

\begin{prop}
There is a functor $F^G$ which completes the following commutative diagram
$$
\xymatrix{\uPair^G(W) \ar[r]^-{\uCok^G} \ar[d]& \uCoh^G(X_0) \ar[d]\\
          \uMF^G(W) \ar[r]^-{F^G} &  \Dd^G_{\Sg}(X_0).}
$$
Moreover, the functor $F^G$ is an exact functor between triangulated categories.
\end{prop}

\begin{proof}
Most of the argument is identical to the non-equivariant case proved in \cite[Proposition~3.7]{O105}. We define a functor $F^G\colon \uPair^G(W) \fl \Dd^G_{\Sg}(X_0)$ to be the composition of $\uCok^G$ and the natural functor from $\uCoh^G(X_0)$ to $\Dd^G_{\Sg}(X_0)$. To prove that $F^G$ induces a functor from $\uMF^G(W)$ to $\Dd^G_{\Sg}(X_0)$ we need to show that any morphism
$\phi=(\phi_0,\phi_1)\colon \overline{P} \fl \overline{Q}$ in $\uPair^G(W)$ which is homotopic to $0$ goes to
$0$-morphism in $\Dd^G_{\Sg}(X_0)$. Fix a homotopy $t=(t_0,t_1)$ where $t_0\colon P_0 \fl Q_1$ and $t_1\colon P_1 \fl Q_0$.
Consider the following decomposition of $\phi$:
$$
\xymatrix{P_1 \ar@<0.4ex>[r]^-{p_1} \ar[d]_-{(t_1,\phi_1)}& P_0
 \ar@<0.4ex>[l]^-{p_0}\ar[d]^-{(t_0,\phi_0)} \ar[r] &\ucoker p_1 \ar[d] \\
 Q_0\oplus Q_1 \ar@<0.4ex>[r]^-{c_1} \ar[d]_-{\mathrm{pr}}& Q_1 \oplus Q_0
 \ar@<0.4ex>[l]^-{c_0}\ar[d]^-{\mathrm{pr}} \ar[r] &Q_0/W \ar[d]\\
 Q_1 \ar@<0.4ex>[r]^-{q_1}& Q_0
 \ar@<0.4ex>[l]^-{q_0} \ar[r] &\ucoker q_1}
$$
where
$$
c_0=\left( \begin{array}{cc} -q_0 & \mathrm{id} \\
                                               0 & q_1 \end{array}  \right), \quad c_1=\left( \begin{array}{cc}
      -q_1 & \mathrm{id} \\
                                                0 & q_0 \end{array}  \right).
$$
This gives a decomposition of $F^G(\phi)$ through a $G$-equivariant locally free object $Q_0/W$ on $X_0$. By Proposition~\ref{prop2.3} we have that $F^G(\phi)=0$ in the
category $\Dd^G_{\Sg}(X_0)$. It is not difficult to check that $F^G$ takes a standard
triangle in $\uMF^G(W)$ to an exact triangle in $\Dd^G_{\Sg}(X_0)$. Therefore $F^G$ is exact.
\end{proof}

Notice that there is a natural forgetful functor $U\colon   \uMF^G(W) \fl  \uMF(W)$, which simply forgets the $G$-action. We have the natural second forgetful functor $U\colon  \Dd^G_{\Sg}(X_0) \fl \Dd_{\Sg}(X_0)$. 
For each $\overline{P}$ in $\uMF^G(W)$, the two objects $UF^G\overline{P}$ and $FU\overline{P}$ coincide. More precisely, there is a commutative diagram
$$
\xymatrix{\uMF^G(W) \ar[r]^-{F^G} \ar[d]_-{U}& \Dd^G_{\Sg}(X_0) \ar[d]^-{U}\\
          \uMF(W) \ar[r]^{\sim}_-{F} &  \Dd_{\Sg}(X_0).}
$$
We can now prove the main result of this section.

\begin{thm}\label{branes-singularities}
The functor $F^G\colon  \uMF^G(W) \fl \Dd^G_{\Sg}(X_0)$ is an equivalence of triangulated categories.
\end{thm}

\begin{proof}
First we verify that the functor $F^G$ is fully faithful. This follows from the arguments of \cite[Lemma~5]{Pl07}. We repeat the proof in the current setting. Fix two objects $\overline{P}$ and $\overline{Q}$ in $\uMF^G(W)$. By definition of
morphisms in $\uMF^G(W)$ and $\Dd^G_{\Sg}(X_0)$, we have a diagram
$$
\xymatrix{\uHom_{\uMF(W)}(U\overline{P},U\overline{Q}) \ar[r]^-{\sim} &
\uHom_{ \Dd_{\Sg}(X_0)}(FU\overline{P},FU\overline{Q})\\
\uHom_{\uMF(W)}(U\overline{P},U\overline{Q})^G \ar@{^{(}->}[u] &
\uHom_{ \Dd_{\Sg}(X_0)}(UF^G\overline{P},UF^G\overline{Q})^G \ar@{^{(}->}[u]\\
\uHom_{\uMF^G(W)}(\overline{P},\overline{Q}) \ar@{=}[u] \ar[r] &
\uHom_{ \Dd^G_{\Sg}(X_0)}(F^G\overline{P},F^G\overline{Q}) \ar@{=}[u]}
$$
and the top morphism is a bijection. Thus the lower map of the diagram is injective, and hence $F^G$ is faithful. To
see that $F^G$ is full as well, consider the following variation of the former diagram
$$
\xymatrix{\uHom_{\uMF(W)}(U\overline{P},U\overline{Q}) \ar[r]^-{\sim}\ar@{->>}[d]_-{\pi} &
\uHom_{ \Dd_{\Sg}(X_0)}(FU\overline{P},FU\overline{Q}) \ar@{->>}[d]^-{\rho}\\
\uHom_{\uMF(W)}(U\overline{P},U\overline{Q})^G  &
\uHom_{ \Dd_{\Sg}(X_0)}(UF^G\overline{P},UF^G\overline{Q})^G \\
\uHom_{\uMF^G(W)}(\overline{P},\overline{Q}) \ar@{=}[u] \ar[r] &
\uHom_{ \Dd^G_{\Sg}(X_0)}(F^G\overline{P},F^G\overline{Q}) \ar@{=}[u]}
$$
using the averaging (or Reynolds) operators $\pi$ and $\rho$. We obviously have $\pi(\phi)=\phi$ (respectively $\rho(f)=f$) if and only if $\phi$
(respectively $f$) is a $G$-equivariant morphism. In particular, $\pi$ and $\rho$ are surjective. The fact that the
functor $F$ is full then implies the same property for $F^G$.

What remains to be proved is that every object $\As$ in $\Dd_{\Sg}^G(X_0)$ is isomorphic to $F^G
\overline{P}$ for some $\overline{P}$.  A complete proof of this is given in  \cite[Theorem~3.9]{O105}; it carries over without change. 
\end{proof}

\section{McKay correspondence for Landau-Ginzburg models}
Here we use the results from the preceding sections to prove a version of the McKay correspondence for Landau-Ginzburg models. We begin by reviewing the basic setting.

Let $M=\Cx^n$ be the complex $n$-dimensional affine space, and let $G$ be a finite subgroup of $\uSL(n,\Cx)$. Put $X=M/G$ and let $\pi\colon  M \fl X$ denote the natural projection. We assume that $G$ acts on $M$ freely outside the origin, which means that $X$ has an isolated singularity\footnote{This is for the purpose of simplicity \textendash the method would seem to be applicable to the general case with some modifications.}. Write $G\text{-}\!\uHilb(M)$ for the Hilbert scheme parametrising $G$-clusters in $M$, that is, the scheme parametrising $G$-invariant subschemes $Z \subset M$ of dimension zero with global sections $H^0(\Os_{Z})$ isomorphic as a $\Cx G$-module to the regular representation of $G$. Let $Y$ be the irreducible component of $G\text{-}\!\uHilb(M)$ which contains the $G$-clusters of free orbits. There is a Hilbert-Chow morphism $\tau: G\text{-}\!\uHilb(M) \fl X$ which, on closed points, sends a $G$-cluster to the orbit supporting it. This morphism is always projective and the irreducible component $Y \subset G\text{-}\!\uHilb(M)$ is mapped birationally onto $X$. We use the same notation $\tau$ for the restriction of the map to $Y$. 

Now let $\Zs \subset Y \times M$ denote the universal closed subscheme, and consider its structure sheaf $\Os_{\Zs}$. We remark that $\Os_{\Zs}$ has finite homological dimension, because $\Os_{\Zs}$ is flat over $Y$ and $M$ is nonsingular. Let $\Dd(Y)$ and $\Dd^G(M)$ denote the bounded derived categories of coherent sheaves on $Y$ and $G$-equivariant coherent sheaves on $M$, respectively. If $\pi_Y$ and $\pi_M$ are the projections from $Y \times M$ to $Y$ and $M$, we define a functor $\Phi\colon \Dd(Y) \fl \Dd^G(M)$ by the formula
$$
\Phi(-)=\Rd \pi_{M *}(\Os^{\vee}_{\Zs}[n] \otimes^{\Ld} \pi_{Y}^*(- \otimes \rho_0))
$$
where $\Os^{\vee}_{\Zs}$ denotes the derived dual $\Rd\Hs om\spdot_{Y \times M}(\Os_{\Zs},\Os_{Y \times M})$. Our main result will be shown under the following assumption.
\begin{assum}\label{assumBKR}
$\tau\colon Y \fl X$ is a crepant resolution and $\Phi$ is an equivalence of triangulated categories.
\end{assum}
The quasi-inverse $\Psi\colon  \Dd^G(M)\fl \Dd(Y)$ can be calculated using Grothendieck duality as the right adjoint of $\Phi$, given by the formula
$$
\Psi(-)=[\Rd \pi_{Y *}(\Os_{\Zs} \otimes^{\Ld} \pi_{M}^*(- ))]^G.
$$
Assumption~\ref{assumBKR} is known to hold if $\dim (Y \times_X Y) \leq n+1$ due to work of Bridgeland, King and Reid \cite{BKR01} together with the results of \cite{BridgMacio02}. In the case of $n \leq 3$, this dimension condition is always fulfilled because the exceptional locus of $\tau$ has dimension $\leq 2$. However, for $n \geq 4$ this condition rarely holds.

We need to make a remark here. In \cite{BKR01}, the definitions of $\Phi$ and $\Psi$ differ slightly from the ones we took. Bridgeland, King and Reid define
\begin{align*}
\Phi(-) &=\Rd \pi_{M *}(\Os_{\Zs} \otimes^{\Ld} \pi_{Y}^*(- \otimes \rho_0)), \\
\Psi(-)&=[\Rd \pi_{Y *}(\Os^{\vee}_{\Zs}[n] \otimes^{\Ld} \pi_{M}^*(- ))]^G.
\end{align*}
It is clear that this difference does not really change the proof of the main result of \cite{BKR01}. The only difference is that everywhere $\Os_{\Zs}$ and $\Os_{\Zs}^{\vee}$ become interchanged.

Assume now that $f\colon M \fl \Cx$ is a regular function with an isolated critical point at the origin which is invariant with respect to the action of $G$ on $M$. We can regard $M$ as a Landau-Ginzburg orbifold with superpotential $f$. We denote by $M_0$ the fiber of the map $f$ over the point $0 \in \Cx$. Next, let $\overline{f}\colon X \fl \Cx$ be the unique morphism such that $f=\overline{f} \pi$. Another Landau-Ginzburg model consists of the variety $Y$ and superpotential $g\colon Y \fl \Cx$ obtained by pullback of $\overline{f}$ to $Y$. We let $Y_0$ be the fiber of $g$ over the point $0$. Note that $Y_0$ contains the exceptional locus $\tau^{-1}(\pi(0))$ of the resolution. Note also that the function $g$ will, generally speaking, have non-isolated critical points. For future use, we let $i_0\colon  Y_0 \fl Y$ and $j_0\colon M_0 \fl M$ denote the corresponding closed immersions of fibers. 

We now head towards proving the main result of this section, which asserts that there is an equivalence between the category of singularities of $Y_0$ and the $G$-equivariant category of singularities of $M_0$. First, however, we must provide preliminary results. Let us denote by $p_Y$ and $p_M$ the projections of the fiber product $Y \times_{\Cx}M$ onto its factors so that we have the following cartesian diagram:
$$
\xymatrix{ & Y \times_{\Cx}M \ar[dr]^-{p_M} \ar[dl]_-{p_Y} & \\
              Y  \ar[dr]_-{g}&  & M \ar[dl]^-{f} \\ 
              &  \Cx  & }
$$
The universal sheaf $\Os_{\Zs}$ on $Y \times M$ is actually supported on the closed subscheme $j\colon  Y \times_{\Cx}M \hookrightarrow Y \times M$. Thus there is a sheaf $\Gs$ on $Y \times_{\Cx}M$, flat over $Y$, such that $\Os_{\Zs}=j_* \Gs$.

Let $\Dd(Y_0)$ denote the bounded derived category of coherent sheaves on $Y_0$ and $\Dd^G(M_0)$ the bounded derived category of $G$-equivariant coherent sheaves on $M_0$. Write $k_0$ for the natural immersion $Y_0 \times M_0 \hookrightarrow Y \times_{\Cx}M$. Then $\Gs_0\spdot=\Ld k_0^* \Gs$ has finite homological dimension and we may define a functor $\Psi_0\colon \Dd^G(M_0) \fl \Dd(Y_0)$ by the formula
$$
\Psi_0(-)=[\Rd \pi_{Y_0 *}(\Gs_0\spdot \otimes^{\Ld} \pi_{M_0}^*(- ))]^G,
$$
where $\pi_{Y_0}$ and $\pi_{M_0}$ are the projections of $Y_0 \times M_0$ to $Y_0$ and $M_0$. That the functor $\Rd \pi_{Y_0 *}(\Gs_0\spdot \otimes^{\Ld}-)$ takes $\Dd^G(Y_0 \times M_0)$ to $\Dd^G(Y_0)$ can easily be seen from the argument of \cite[Lemma~2.1]{CH02} since the support of $\Gs_0\spdot$ is proper over $Y_0$.

We obtain a useful and probably well-known result, a version of which can be found in \cite[Lemma~6.1]{CH02}.

\begin{lem}\label{lem7.3}
There is a natural isomorphism of functors:
$$
i_{0 *}\Psi_0(-) \cong \Psi j_{0 *}(-).
$$
\end{lem}

\begin{proof}
We first note that there exist a natural isomorphism between the functors
$$
\xymatrix@C-1mm{\Dd^G(Y_0) \ar[r]^-{[-]^G} & \Dd(Y_0) \ar[r]^-{i_{0 *}} & \Dd(Y)}
$$
and
$$
\xymatrix@C-1mm{\Dd^G(Y_0) \ar[r]^-{i_{0 *}} &\Dd^G(Y) \ar[r]^-{[-]^G} & \Dd(Y)}.
$$
The cartesian diagram
$$
\xymatrix{Y_0 \times M_0 \ar[r]^-{k_0} \ar[d]_-{\pi_{M_0}} & Y \times_{\Cx} M \ar[d]^-{p_M} \\
               M_0 \ar[r]_-{j_0} & M}
$$
shows that
$$
k_{0 *}\pi_{M_0}^*(- ) \cong p_{M}^*j_{0*}(- ) \cong  \Ld j^* \pi_M^*j_{0*}(- ).
$$
By the projection formula, we can then write
\begin{align*}
j_* k_{0*}(\Ld k_0^*\Gs \otimes^{\Ld} \pi_{M_0}^*(- )) &\cong j_* (\Gs \otimes^{\Ld} k_{0 *}\pi_{M_0}^*(- )) \\
& \cong j_* (\Gs \otimes^{\Ld} \Ld j^* \pi_M^*j_{0*}(- )) \\
& \cong j_*\Gs \otimes^{\Ld}  \pi_M^*j_{0*}(- ) \\
& \cong \Os_{\Zs} \otimes^{\Ld}  \pi_M^*j_{0*}(- ).
\end{align*}
Putting these observations together, we obtain the desired isomorphism:
\begin{align*}
i_{0 *}\Psi_0(-) &= i_{0*}[\Rd \pi_{Y_0 *}(\Gs_0\spdot \otimes^{\Ld} \pi_{M_0}^*(- ))]^G \\
&\cong [i_{0*}\Rd \pi_{Y_0 *}(\Gs_0\spdot \otimes^{\Ld} \pi_{M_0}^*(- ))]^G \\
&\cong [\Rd \pi_{Y*}(i_0 \times j_0)_*(\Gs_0\spdot \otimes^{\Ld} \pi_{M_0}^*(- ))]^G \\
&\cong [\Rd \pi_{Y*}j_* k_{0*}(\Ld k_0^*\Gs \otimes^{\Ld} \pi_{M_0}^*(- ))]^G \\
&\cong [\Rd \pi_{Y*}(\Os_{\Zs} \otimes^{\Ld}  \pi_M^*j_{0*}(- ))]^G \\
&=\Psi j_{0 *}(-). \qedhere
\end{align*}
\end{proof}

We want now to consider a correspondence in the opposite direction. The main problem is the right adjoint to $\Rd \pi_{Y_0*}$ as $\pi_{Y_0}$ is manifestly non-proper. However, using Deligne's construction of $\pi_{Y_0}^!$ in the context of general Grothendieck duality theory (cf.~\cite{Del66,Sas04,Lip08,Log08}) we can still obtain a right adjoint to $\Rd \pi_{Y_0*}$ for the full subcategory of $\Dd^G(Y_0 \times M_0)$ consisting of objects whose support is proper over $Y_0$. Let us see how this comes about.

Let $\overline{M}_0$ be the closure of $M_0$ in the projective space $\Px^n$. Then the map $\pi_{Y_0}$ factorizes as $\pi_{Y_0}=\overline{\pi}_{Y_0} \iota$ where $\iota\colon  Y_0 \times M_0 \hookrightarrow Y_0 \times \overline{M}_0$ is an open immersion and $\overline{\pi}_{Y_0}\colon Y_0 \times \overline{M}_0 \fl Y_0$ is the projection. In this way we get an extension of $\pi_{Y_0}$ which is a proper map. Now define the functor $\pi_{Y_0}^!\colon  \Dd(Y_0) \fl \Dd(Y_0 \times M_0)$ to be $\iota^* \overline{\pi}_{Y_0}^!$. A reasoning as in \cite[Lemma~4]{Log08} shows that there is a functorial isomorphism
$$
\uHom_{\Dd(Y_0)}(\Rd \pi_{Y_0*} \Es\spdot,\Fs\spdot) \cong \uHom_{\Dd(Y_0 \times M_0)}(\Es\spdot,\pi_{Y_0}^!\Fs\spdot),
$$ 
for every object $\Es\spdot$ in $\Dd(Y_0 \times M_0)$ whose support is proper over $Y_0$ and any $\Fs\spdot$ in $\Dd(Y_0)$. Furthermore, since the map $\pi_{Y_0}$ is of finite Tor-dimension and of finite type, it follows from \cite[Theorem~4.9.4]{Lip08} that there is a functorial isomorphism 
$$
\pi_{Y_0}^! \Fs\spdot \cong \pi_{Y_0}^! \Os_{Y_0} \otimes^{\Ld}\pi_{Y_0}^* \Fs\spdot,
$$
for any $\Fs\spdot \in \Dd(Y_0)$. Let us remark that the above extends straightforwardly to the corresponding $G$-equivariant categories. 

Let $\Phi_0\colon  \Dd(Y_0) \fl \Dd^G(M_0)$ denote the functor in the other direction defined as
$$
\Phi_0(-)=\Rd \pi_{M_0 *}\Rd\Hs om\spdot_{Y_0 \times M_0}(\Gs_0\spdot, \pi_{Y_0}^!(- \otimes \rho_0)).
$$
Observe that the fact that $\tau$ is proper implies that the support of $\Gs_0\spdot$ is proper over $M_0$. Arguing as before one can check that $\Rd \pi_{M_0 *}\Rd\Hs om\spdot_{Y_0 \times M_0}(\Gs_0\spdot,-)$ sends $\Dd^G(Y_0 \times M_0)$ to $\Dd^G(M_0)$, so $\Phi_0$ is well-defined. 

The following is an immediate consequence of the definition.

\begin{lem}
$\Phi_0$ is right adjoint to $\Psi_0$.
\end{lem}

\begin{proof}
Indeed, for any $\Es\spdot \in \Dd(Y_0)$ and $\Fs\spdot \in\Dd^G(M_0)$ one has a sequence of isomorphisms:
\begin{align*}
 &\uHom_{\Dd^G(M_0)} (\Fs\spdot, \Phi_0 \Es\spdot) \\
&\quad = \uHom_{\Dd^G(M_0)} (\Fs\spdot,\Rd \pi_{M_0 *}\Rd\Hs om\spdot_{Y_0 \times M_0}(\Gs_0\spdot, \pi_{Y_0}^!(\Es\spdot \otimes \rho_0))) \\
&  \quad \cong \uHom_{\Dd^G(Y_0 \times M_0)} (\pi_{M_0}^*\Fs\spdot,\Rd\Hs om\spdot_{Y_0 \times M_0}(\Gs_0\spdot, \pi_{Y_0}^!(\Es\spdot \otimes \rho_0))) \\
&  \quad \cong \uHom_{\Dd^G(Y_0 \times M_0)}(\Gs_0\spdot \otimes^{\Ld} \pi_{M_0}^*\Fs\spdot, \pi_{Y_0}^!(\Es\spdot \otimes \rho_0))\\
& \quad \cong \uHom_{\Dd^G(Y_0)}(\Rd \pi_{Y_0 *}(\Gs_0\spdot \otimes^{\Ld} \pi_{M_0}^*\Fs\spdot), \Es\spdot \otimes \rho_0)\\
& \quad \cong \uHom_{\Dd(Y_0 )}([\Rd \pi_{Y_0 *}(\Gs_0\spdot \otimes^{\Ld} \pi_{M_0}^*\Fs\spdot)]^G, \Es\spdot) \\
& \quad \cong \uHom_{\Dd(Y_0 )}(\Psi_0 \Fs\spdot, \Es\spdot).
\end{align*} 
Here, the third isomorphism is the aforementioned duality for $\pi_{Y_0}$, which can be applied since $\Gs_0\spdot$ has proper support over $Y_0$.
\end{proof}

We now make an observation to be applied in the subsequent argument.

\begin{lem}\label{lem7.5}
There is an isomorphism:
$$
\Ld k_0^* \Rd\Hs om\spdot_{Y \times_{\Cx} M}(\Gs, p_Y^! \Os_{Y}) \cong \Rd\Hs om\spdot_{Y_0 \times M_0}(\Gs_0\spdot, \pi_{Y_0}^! \Os_{Y_0}).
$$
\end{lem}

\begin{proof}
We have to prove that the natural morphism
$$
\Ld k_0^* \Rd\Hs om\spdot_{Y \times_{\Cx} M}(\Gs, p_Y^! \Os_{Y}) \lfl \Rd\Hs om\spdot_{Y_0 \times M_0}(\Ld k_0^* \Gs, \pi_{Y_0}^! \Os_{Y_0})
$$
is an isomorphism. Since $k_0$ is a closed immersion, it is enough to prove that the induced morphism
$$
k_{0*}\Ld k_0^* \Rd\Hs om\spdot_{Y \times_{\Cx} M}(\Gs, p_Y^! \Os_{Y}) \lfl k_{0*}\Rd\Hs om\spdot_{Y_0 \times M_0}(\Ld k_0^* \Gs, \pi_{Y_0}^! \Os_{Y_0})
$$
is an isomorphism. Consider the cartesian diagram
$$
\xymatrix{Y_0 \times M_0 \ar[r]^-{k_0} \ar[d]_-{\pi_{Y_0}} & Y \times_{\Cx} M \ar[d]^-{p_Y} \\
               Y_0 \ar[r]_-{i_0} & Y.}
$$
We have $k_{0*}\pi_{Y_0}^*(-) \cong p_Y^* i_{0 *}(-)$. By the projection formula, we deduce that the first member is isomorphic to
\begin{align*}
 &\Rd\Hs om\spdot_{Y \times_{\Cx} M}(\Gs, p_Y^! \Os_{Y}) \otimes^{\Ld} k_{0*} \Os_{Y_0 \times M_0} \\
&\qquad\cong  \Rd\Hs om\spdot_{Y \times_{\Cx} M}(\Gs, p_Y^! \Os_{Y}) \otimes^{\Ld} k_{0*}\pi_{Y_0}^* \Os_{Y_0} \\
&\qquad\cong  \Rd\Hs om\spdot_{Y \times_{\Cx} M}(\Gs, p_Y^! \Os_{Y}) \otimes^{\Ld} p_Y^* i_{0 *} \Os_{Y_0} \\
&\qquad\cong  \Rd\Hs om\spdot_{Y \times_{\Cx} M}(\Gs, p_Y^! i_{0 *} \Os_{Y_0})
\end{align*}
where the last step follows from the observation that $\Gs$ has finite homological dimension. The second member is isomorphic to
$$
\Rd\Hs om\spdot_{Y \times_{\Cx} M}( \Gs, k_{0*} \pi_{Y_0}^! \Os_{Y_0})
$$
by the adjoint property of $\Ld k_0^*$ and $k_{0*}$. Thus, we have to prove that the natural morphism
$$
 \Rd\Hs om\spdot_{Y \times_{\Cx} M}(\Gs, p_Y^! i_{0 *} \Os_{Y_0}) \lfl \Rd\Hs om\spdot_{Y \times_{\Cx} M}( \Gs, k_{0*} \pi_{Y_0}^! \Os_{Y_0})
$$
is an isomorphism. Then, it is enough to see that $p_Y^! i_{0 *} \Os_{Y_0} \cong k_{0*} \pi_{Y_0}^! \Os_{Y_0}$. This follows from the isomorphisms
\begin{align*}
\uHom_{\Dd(Y \times_{\Cx} M)}(\Es\spdot,p_Y^! i_{0 *} \Os_{Y_0} ) & \cong \uHom_{\Dd(Y)}(\Rd p_{Y*}\Es\spdot, i_{0 *} \Os_{Y_0} ) \\
& \cong \uHom_{\Dd(Y_0)}(\Ld i_{0}^*\Rd p_{Y*}\Es\spdot,\Os_{Y_0} ) \\
&  \cong \uHom_{\Dd(Y_0)}(\Rd \pi_{Y_0*}\Ld k_0^* \Es\spdot,\Os_{Y_0} ) \\
&  \cong \uHom_{\Dd(Y_0 \times M_0)}(\Ld k_0^* \Es\spdot,\pi_{Y_0}^!\Os_{Y_0} )\\
&  \cong \uHom_{\Dd(Y \times_{\Cx}M)}( \Es\spdot,k_{0*}\pi_{Y_0}^!\Os_{Y_0} )
\end{align*}
which hold for any object $\Es\spdot$ in  $\Dd(Y \times_{\Cx}M)$ whose support is proper over $Y$ (here we used the base change theorem for the above cartesian diagram; see \cite[Sect.~1]{HLS06}).
\end{proof}

Before stating our next result, it will be convenient to provide the following piece of information. As we pointed out earlier, there exist a functorial isomorphism 
$\pi_{Y_0}^!(-) \cong \pi_{Y_0}^! \Os_{Y_0} \otimes^{\Ld}\pi_{Y_0}^*(-)$. Using the fact that $\Gs\spdot_0$ has finite homological dimension we obtain an isomorphism
\begin{align*}
\Rd\Hs om\spdot_{Y_0 \times M_0}(\Gs_0\spdot, \pi_{Y_0}^!(-)) & \cong  \Rd\Hs om\spdot_{Y_0 \times M_0}(\Gs_0\spdot, \pi_{Y_0}^! \Os_{Y_0} \otimes^{\Ld}\pi_{Y_0}^*(-)) \\
& \cong \Rd\Hs om\spdot_{Y_0 \times M_0}(\Gs_0\spdot, \pi_{Y_0}^! \Os_{Y_0}) \otimes^{\Ld} \pi_{Y_0}^*(- ).
\end{align*}
Thus, denoting $\Ks_0\spdot=\Rd\Hs om\spdot_{Y_0 \times M_0}(\Gs_0\spdot, \pi_{Y_0}^! \Os_{Y_0})$, we can rewrite $\Phi_0$ as
$$
\Phi_0(-) \cong \Rd \pi_{M_0 *}(\Ks_0\spdot \otimes^{\Ld} \pi_{Y_0}^*(- \otimes \rho_0)).
$$
Combining these remarks with Lemma \ref{lem7.5} we have the following. 

\begin{lem}\label{lem7.6}
There is a natural isomorphism of functors:
$$
j_{0 *}\Phi_0(-) \cong \Phi i_{0 *}(-).
$$
\end{lem}

\begin{proof}
The argument is very similar to that used in the proof of Lemma~\ref{lem7.3}. We give it for the sake of completeness. To begin with, we observe that there is a natural isomorphism between the functors
$$
\xymatrix{\Dd(Y_0) \ar[r]^-{-\otimes \rho_0} &\Dd^G(Y_0) \ar[r]^-{i_{0 *}} & \Dd^G(Y)}
$$
and
$$
\xymatrix{\Dd(Y_0) \ar[r]^-{i_{0 *}} & \Dd(Y) \ar[r]^-{-\otimes \rho_0} & \Dd^G(Y).}
$$
Invoking Lemma \ref{lem7.5} and the projection formula, we obtain that
\begin{align*}
&j_* k_{0*}(\Ks\spdot_0 \otimes^{\Ld} \pi_{Y_0}^*(- \otimes \rho_0))\\
& \qquad \cong  j_* k_{0*}(\Ld k_0^* \Rd\Hs om\spdot_{Y \times_{\Cx} M}(\Gs, p_Y^! \Os_{Y}) \otimes^{\Ld} \pi_{Y_0}^*(- \otimes \rho_0))\\
& \qquad \cong j_* ( \Rd\Hs om\spdot_{Y \times_{\Cx} M}(\Gs, p_Y^! \Os_{Y}) \otimes^{\Ld}k_{0*}\pi_{Y_0}^*(- \otimes \rho_0)) \\
& \qquad \cong j_* ( \Rd\Hs om\spdot_{Y \times_{\Cx} M}(\Gs, p_Y^! \Os_{Y}) \otimes^{\Ld}p_{Y}^*i_{0*}(- \otimes \rho_0)) \\
&  \qquad \cong j_* ( \Rd\Hs om\spdot_{Y \times_{\Cx} M}(\Gs, p_Y^! \Os_{Y}) \otimes^{\Ld}\Ld j^* \pi_{Y}^*i_{0*}(- \otimes \rho_0)) \\
& \qquad \cong j_* \Rd\Hs om\spdot_{Y \times_{\Cx} M}(\Gs,  p_Y^!  \Os_{Y}) \otimes^{\Ld} \pi_{Y}^*i_{0*}(- \otimes \rho_0).
\end{align*}
On the other hand, by relative Grothendieck duality, we get
\begin{align*}
j_* \Rd\Hs om\spdot_{Y \times_{\Cx} M}(\Gs,  p_Y^!  \Os_{Y}) & \cong j_* \Rd\Hs om\spdot_{Y \times_{\Cx} M}(\Gs, j^! \pi_Y^!  \Os_{Y})\\
& \cong \Rd\Hs om\spdot_{Y \times M}(j_*\Gs, \pi_Y^! \Os_{Y}) \\
& \cong \Rd\Hs om\spdot_{Y \times M}(\Os_{\Zs}, \pi_Y^! \Os_{Y})\\
& \cong \Os^{\vee}_{\Zs}[n],
\end{align*}
where we used the isomorphism $\pi_Y^! \Os_{Y} \cong \Os_{Y \times M}[n]$ which follows from the triviality of the canonical bundle $\omega_M$.
Hence
\begin{align*}
j_* k_{0*}(\Ks\spdot_0 \otimes^{\Ld} \pi_{Y_0}^*(- \otimes \rho_0)) \cong \Os^{\vee}_{\Zs}[n] \otimes^{\Ld} \pi_{Y}^*i_{0*}(- \otimes \rho_0).
\end{align*}
Wrapping things up, we conclude that
\begin{align*}
j_{0 *}\Phi_0(-) &= j_{0*}\Rd \pi_{M_0 *}(\Ks\spdot_0 \otimes^{\Ld} \pi_{Y_0}^*(- \otimes \rho_0)) \\
&\cong \Rd \pi_{M*}(i_0 \times j_0)_*(\Ks\spdot_0 \otimes^{\Ld} \pi_{Y_0}^*(- \otimes \rho_0)) \\
&\cong  \Rd \pi_{M*}j_* k_{0*}(\Ks\spdot_0 \otimes^{\Ld}\pi_{Y_0}^*(- \otimes \rho_0)) \\
&\cong  \Rd \pi_{M*}(\Os^{\vee}_{\Zs}[n] \otimes^{\Ld} \pi_{Y}^*(i_{0*}(-) \otimes \rho_0)) \\
&=\Phi i_{0*}(-),
\end{align*}
as asserted.
\end{proof}

The following result is the goal we have been striving for throughout this whole section.

\begin{thm}\label{thmMcKayLG}
Under Assumption~\ref{assumBKR}, the functors $\Phi_0$ and $\Psi_0$ define inverse equivalences between $\Dd(Y_0)$ and $\Dd^G(M_0)$. These equivalences induce equivalences $\overline{\Phi}_0$ and $\overline{\Psi}_0$ between $\Dd_{\Sg}(Y_0)$ and $\Dd_{\Sg}^G(M_0)$.
\end{thm}

\begin{proof}
Let us prove that the composition $\Phi_0 \Psi_0$ is isomorphic to the identity functor on $\Dd^G(M_0)$. The composition in the different order is computed similarly. Consider an object $\Es\spdot \in\Dd^G(M_0)$ and denote the cone of the adjunction morphism $\Es\spdot \fl \Phi_0 \Psi_0\Es\spdot$ by $\Fs\spdot$. Applying $j_{0*}$ yields an exact triangle
$$
\xymatrix@C-1mm{j_{0*}\Es\spdot \ar[r] & j_{0*}\Phi_0 \Psi_0\Es\spdot \ar[r] & j_{0*}\Fs\spdot \ar[r] & j_{0*}\Es\spdot[1]}.
$$
Combining Lemma~\ref{lem7.6} and Lemma~\ref{lem7.3} we get
$$
j_{0*}\Phi_0 \Psi_0(-) \cong \Phi i_{0*}\Psi_0(-) \cong \Phi \Psi j_{0*}(-).
$$
Hence, $j_{0*}\Fs\spdot$ is isomorphic to the cone of the morphism $j_{0*}\Es\spdot \fl  \Phi \Psi j_{0*}\Es\spdot$. Since $\Phi$ is an equivalence and $j_0$ is a closed immersion, one obtains $\Fs\spdot \cong 0$. The conclusion is that the adjunction morphism $\Es\spdot \fl \Phi_0 \Psi_0\Es\spdot$ is an isomorphism.

We next show that the functors $\Phi_0$ and $\Psi_0$ induce equivalences between $\Dd_{\Sg}(Y_0)$ and $\Dd_{\Sg}^G(M_0)$. Let us first make an observation. Let $\Es\spdot$ be a perfect complex on $Y_0 \times M_0$ and let us consider the object $\Rd \pi_{M_0*}(\Ks_0\spdot \otimes^{\Ld}\Es\spdot)$ in the derived category of coherent sheaves on $M_0$. We claim that $\Rd \pi_{M_0*}(\Ks_0\spdot \otimes^{\Ld}\Es\spdot)$ is a perfect complex on $M_0$. To substantiate this claim, it suffices to verify that $\Rd \pi_{M_0*}(\Ks_0\spdot \otimes^{\Ld}\Es\spdot) \otimes^{\Ld}\Fs\spdot$ is an object of $\Dd(M_0)$ for every $\Fs\spdot$ in $\Dd(M_0)$ (see, e.g.~\cite[Lemma~1.2]{HLS07}). But this follows at once from the projection formula for the morphism $\pi_{M_0}$. Similarly, we check that $\Rd \pi_{Y_0*}(\Gs_0\spdot \otimes^{\Ld} \Es\spdot)$ is a perfect complex on $Y_0$. The same situation prevails in the equivariant setting. 

Now, the functors $\pi_{Y_0}^*(-\otimes \rho_0)$ and $\pi_{M_0}^*$ are exact and take perfect complexes to perfect complexes. By what we have just seen, the functors  $\Rd \pi_{M_0*}(\Ks_0\spdot\otimes^{\Ld}-)$
and $[\Rd \pi_{Y_0*}(\Gs_0\spdot \otimes^{\Ld} -)]^G$ also preserve perfect complexes. Hence, owing to Lemma~\ref{lem3.1}, we obtain a functor $\overline{\Phi}_0\colon  \Dd_{\Sg}(Y_0) \fl \Dd_{\Sg}^G(M_0)$ and this functor has the left adjoint $\overline{\Psi}_0\colon  \Dd_{\Sg}^G(M_0) \fl \Dd_{\Sg}(Y_0)$. As the composition $\Phi_0 \Psi_0$ is isomorphic to the identity functor, the composition $\overline{\Phi}_0 \overline{\Psi}_0$ is also isomorphic to the identity functor on $ \Dd_{\Sg}^G(M_0)$. A similar argument shows that the composition  $\overline{\Psi}_0 \overline{\Phi}_0$ is isomorphic to the identity functor on $ \Dd_{\Sg}(Y_0)$. The result then follows immediately.
\end{proof}

It seems appropriate to conclude by examining the implications of this result in the specific context of Section~\ref{phys-arg}. Let $G=\Zx_n$ be a cyclic group in $\uSL(n,\Cx)$ acting on $M=\Cx^n$ and let $Y$ be the canonical crepant resolution of the quotient $X=M/G$. Explicitly we choose coordinates $x_1,\dots,x_n$ on $M$ in terms of which the action of the generator in $G$ is given by $(x_1,\dots,x_n) \mapsto (\varepsilon x_1,\dots,\varepsilon x_n)$ where $\varepsilon= \exp (2 \pi i /n)$ is a fixed $n$th root of unity. The space $Y$ is the blow up of the unique singular point of $X$. It can be described explicitly as follows. Write $P=\Px^{n-1}$ for the projective space with homogeneous coordinates  $x_1,\dots,x_n$. Then $Y=\utot(\Os_{P}(-n))$ is the total space of the line bundle $\Os_{P}(-n)$ and the natural map $\tau\colon  Y \fl X$ is simply contracting the zero section. Let $\Zs \subset Y \times M$ denote the fiber product of $Y$ and $M$ over $X$. Then $\Zs$ can be identified with the total space $\Zs=\utot(\Os_{P}(-1))$ and the map $q\colon \Zs \fl M$ is again the contraction of the zero section, this time to a smooth point  \textendash the origin $0 \in M$. All this data can be conveniently organized in the commutative diagram
$$
\xymatrix@C-1mm@R-3mm{ & \Zs  \ar[ld]_-{\zeta} \ar[dd]_-{p} \ar[rr]^-{q} & & M \ar[dd]^-{\pi} \\
             P &  & \\              
               & Y \ar[lu]^-{\eta} \ar[rr]_-{\tau} &  & X}
$$
where $\eta\colon Y \fl P$ and $\zeta\colon \Zs \fl P$ denote the natural projections and $p\colon  \Zs \fl Y$ is the map of taking a quotient by $G$. Note that the group $G$ acts on $\Zs$ by simply multiplying by $\varepsilon$ along the fibers of $\Os_{P}(-1) \fl P$ and so the map $p\colon  \Zs \fl Y$ can also be viewed as the map raising into $n$th power along the fibers of the line bundle $\Os_{P}(-1)$. Conversely, we can view $\Zs$ as the canonical $n$th root cover of $Y$ which is branched along the zero section $Q \subset Y$ of $\eta$.

Now let $f \in \Cx[x_1,\dots,x_n]$ be a homogeneous polynomial of degree $n$. Then $f$ can be viewed as a regular function on $M$ with a critical point at the origin which is invariant with respect to the action of $G$ on $M$. This way, we get a singular Landau-Ginzburg model $(M,f)$ with an action of $G$.  Let $S$ be the hypersurface of degree $n$ in $P=\Px^{n-1}$ given by the homogeneous equation $f=0$. Consider the associated affine cone over $S$, namely, the hypersurface $M_0$ given in $M=\Cx^n$ by exactly the same equation $f=0$. It is evident that the singular fiber of the map $f\colon M \fl \Cx$ over the point $0 \in \Cx$ is precisely $M_0$. Let $g\colon Y \fl \Cx$ be defined as before, and let $Y_0$ denote the fiber of $g$ over the point $0$. Then $Y_0$ is a normal crossing variety with irreducible components $Y'_0$ and $Y''_0$. One component $Y'_0$ is isomorphic to the total space of the line bundle $\Os_P(-n)\vert_S$ over $S$. The second component $Y''_0$ is isomorphic to $Q$.

It is proved in \cite[Proposition~2.40]{Bl07} that $G\text{-}\!\uHilb(M)$ is isomorphic to $Y$. Moreover, the tautological bundles on $G\text{-}\!\uHilb(M)$ (see \cite{Re96} for the definition) are $\eta^* \Os_P, \eta^* \Os_P(1),\dots,\eta^* \Os_P(n-1)$. It then follows from \cite[Example~4.3]{Bridg05} that $\Phi$ is an equivalence of categories and  we can apply Theorem~\ref{thmMcKayLG} to obtain $\Dd_{\Sg}(Y_0) \cong \Dd^G_{\Sg}(M_0)$. We are now set to establish the claim made at the end of Section~\ref{phys-arg}.

\begin{cor}
Let the context be as above. Then the category of D-branes in the Landau-Ginzburg model $(Y,g)$ is equivalent to the category of D-branes in the Landau-Ginzburg orbifold $(M,f)$. 
\end{cor}

\begin{add}
In a recent preprint P.~Seidel \cite{Seid08a} gave another example illustrating the use of Theorem~\ref{thmMcKayLG} in the context of Homological Mirror Symmetry. 
\end{add}


\addcontentsline{toc}{section}{Bibliography}

\end{document}